\documentclass[11pt,reqno]{amsart}

\usepackage{enumerate}
\usepackage{graphicx,graphics}
\usepackage{amsfonts}
\usepackage{amssymb}
\usepackage{amsthm}
\usepackage{amsmath}
\usepackage{mathrsfs,color}
\usepackage{listings}
\input{xy}
\xyoption{all}

\definecolor{dkgreen}{rgb}{0,0.6,0}
\definecolor{gray}{rgb}{0.5,0.5,0.5}
\definecolor{mauve}{rgb}{0.58,0,0.82}
\newtheorem{theorem}{Theorem}[section]
\newtheorem{lemma}[theorem]{Lemma}

\newtheorem{remark}[theorem]{Remark}
\newtheorem{corollary}[theorem]{Corollary}

\newtheorem{proposition}[theorem]{Proposition}
\newtheorem{example}[theorem]{Example}
\newtheorem{definition}[theorem]{Definition}

\begin{document}

\title[Lattice Lipschitz maps on Banach function spaces]{Lattice Lipschitz superposition operators on Banach function spaces}

\author[R. Arnau, J. M. Calabuig, E. Erdo\u{g}an, E.  A. S\'{a}nchez P\'{e}rez]{ Roger Arnau$^*$, Jose M. Calabuig, Ezgi Erdo\u{g}an, \\ Enrique A. S\'{a}nchez P\'{e}rez}

\address{Ezgi Erdo\u{g}an.
              Department of Mathematics, Faculty of Sciences,  University of Marmara, 34722, Kad\i k\"{o}y, Istanbul, Turkey.}
\email{ ezgi.erdogan@marmara.edu.tr}

\address{ Roger Arnau$^*$, Jose M. Calabuig,
        Enrique \ A.\ S\'{a}nchez P\'{e}rez\\ {$^*$}Corresponding author.
              Instituto Universitario de Matem\'atica Pura y Aplicada,
              Universitat Polit\`ecnica de Val\`encia, Camino de Vera s/n, 46022
              Valencia, Spain.}
            \email{ararnnot@posgrado.upv.es, jmcalabu@mat.upv.es, easancpe@mat.upv.es}



\thanks{
The first author was supported by a contract of the Programa de Ayudas de Investigaci\'on y Desarrollo (PAID-01-21), Universitat Polit\`ecnica de Val\`encia.
This publication is part of the R\&D\&I project/grant PID2022-138342NB-I00 funded by MCIN/ AEI/10.13039/501100011033/ (Spain).}

\maketitle

\begin{abstract} 
We analyse and characterise the notion of lattice Lipschitz operator (a class of superposition operators, diagonal Lipschitz maps) when defined between Banach function spaces. After showing some general results, we restrict our attention to the case of those Lipschitz operators which are representable by pointwise composition with a strongly measurable function. Mimicking the classical definition and characterizations of (linear) multiplication operators between Banach function spaces, we show that under certain conditions the requirement for a diagonal Lipschitz operator to be well-defined between two such spaces $X(\mu)$ and $Y(\mu)$  is that it can be represented by a strongly measurable function
which belongs to the Bochner space $\mathcal M(X,Y) \big(\mu, Lip_0(\mathbb R) \big). $  Here, $\mathcal M(X,Y) $ is the space of multiplication operators between $X(\mu)$ and $Y(\mu),$ and $Lip_0(\mathbb R)$ is the space of real-valued Lipschitz maps with real variable that are equal to $0$ in $0. $ This opens the door to a better understanding of these maps, as well as finding the relation of these operators to some normed tensor products and other classes of maps.
\end{abstract}

\section{Introduction and notation}

Lipschitz functions, which appeared as a tool in mathematical analysis, have found relevant applications in the field of differential equations from the beginning of the 20th century. 
Today, they have become a fundamental tool also in many fields of applied mathematics and artificial intelligence. Recently, more attention has returned to this topic from pure analysis, in a stream of research that tries to extend linear and multilinear issues and methods to the field of Lipschitz functions \cite{ang,cha,mas,farjo}: summability of operators \cite{bel,sal,ya}, ideals of Lipschitz operators \cite{ach0,ach,cab,cha2,mez}, lattice geometric properties of Lipschitz maps \cite{bra,chen} and many others are topics of current interest to the mathematical community.

Following this trend, we are interested in the Lipschitz version of diagonal maps (which appeared in the Euclidean space context) or  multiplication operators (when working in the setting of  function spaces). As can be easily seen this diagonal property coincides with a pointwise domination requirement, for which an order in the space is needed; since we are interested in the setting of the Banach function spaces, the requirement on the order can be easily formulated to provide the (equivalent) notion of lattice Lipschitz operator.  No explanation is needed when we state that diagonalization of operators between finite dimensional spaces, and also multiplication operators, apply in a vast number of pure and applied contexts, and extensions of these genuine linear concepts to other classes of operators are fully justified \cite{fucik,furi}. Thus, our motivation for this research is twofold. On the one hand, we want to study Lipschitz diagonal operators to provide tools to extend the fundamental factorization theorems for operators  between function spaces, which are now at the core of functional analysis (\cite{DieJarTon}, see the last section of the article). On the other hand, guided by some recent applications of Lipschitz functions in artificial intelligence algorithms (mainly in reinforcement learning \cite{asa,cal,fal} but also in some separation techniques with applications in clustering procedures \cite{von}), we want to increase the mathematical tools for this applied research.

The family of functions we are interested in are a particular case of what are called superposition operators \cite{appel,appel2}. Since the eighties of the last century there have been some specific investigations on different properties of Lipschitz-type operators that are also superposition operators. In fact, some relevant characterizations of locally Lipschitz superposition operators involving inequalities such as the ones we analyze here can be found in some papers on the subject (see for example \cite[Th. 1]{appel}). However, although these inequalities are interesting for analyzing the properties of Lipschitz superposition operators, in this article we face the problem of the functional representation of these maps. In other words, in this paper we are interested in finding some general conditions to ensure that a given lattice Lipschitz operator can be represented as a  strongly measurable vector function belonging to some special K\"othe-Bochner space, by using a specific definition of how a vector function acts on a scalar function belonging to a  Banach function space.

Thus, we will devote the second part of the article to analyze how to relate lattice Lipschitz operators and strongly measurable vector valued  functions. In the first part, we will present the main results on the structure of lattice Lipschitz maps and diagonal Lipschitz operators on Banach function spaces. These spaces are the natural extension of classical Lebesgue spaces (but also of Euclidean spaces), and include, for example, Lorentz spaces of sequences and functions, as well as Orlicz spaces. The direct translation of the notion of lattice Lipschitz operator, which has originally been done in Euclidean spaces, gives a rather vague concept, which we show could include some ``pathological'' maps, which appear essentially due to the fact that the space of real-valued Lipschitz operators $Lip_0(\mathbb R)$ is not separable. To ensure a good behavior of the maps involved we will focus on the case of superposition operators that allow a representation as a pointwise composition with a strongly measurable function. 


Starting from the description of the notion of  ``diagonal'' operator between Banach function spaces in the linear case (Section 2), we will show some general properties, examples and counterexamples (Section 3), and, in the following Section 4 we will focus the attention on lattice Lipschitz operators that allow representations by strongly measurable functions, showing our main representation theorems. Finally, in Section 5 we show some applications, in connection with some classical normed tensor products and operator ideals.

Let us write the direct translation of the notion of lattice Lipschitz operator that has been originally given in the framework of Euclidean spaces (\cite[Definition 1]{arn}), to the setting of the Banach lattices. In the context of the paper \cite{arn}, we say that a Lipschitz operator $T:L \to V$ on the Euclidean space $L$ is lattice Lipschitz if there is a constant $K>0$ such that  for every pair of elements $x,y \in L,$ 
$$
\big| T(x) - T(y) \big| \le   K \big| x- y \big|.
$$
Other notions related to this definition can be found in the literature on maps on Banach lattices (see for example \cite[Ch.3]{cobzas}, \cite{jiang_li}).

\vspace{0.5cm}

Before getting into the subject, let us recall some basic definitions and results from the theory of Banach function spaces and general functional analysis.
Throughout the paper  $(\Omega,\Sigma,\mu)$ will be a finite measure space.
Every representation of a simple function $f=\sum_{i=1}^n \lambda_i \chi_{A_i}$ is assumed to be disjoint, that is $A_i \cap A_j = \emptyset$ for $i \ne j.$
We assume that all the measures in the current manuscript are complete, meaning that every subset of a $\mu$-null (measurable) set is again measurable.

Given a Banach space $E$, let us review some basic definitions about measurability of functions from $\Omega$ to $X$.
A simple function is a function as $\psi = \sum_{i=1}^n x_i \chi_{A_i},$ for $x_1,...,x_n \in E$ and disjoint $A_1,...,A_n \in \Sigma.$
A vector valued function $\Psi: \Omega \to E$ is said to be weakly measurable if the real function given by the  composition $x^* \circ \Psi$ is measurable for every $x^* \in E^*.$
$\Phi$ is strongly measurable if it equals almost everywhere the limit of a sequence of simple functions,  i.e.,
\begin{equation*}
	\lim_n \| \Psi(w) - \psi_n(w)\|_X = 0 \quad \text{ for almost every } \quad w \in \Omega.
\end{equation*}

A function $\Psi:\Omega \to E$ is essentially separably valued if there is a $\mu-$null set $ A \in \Sigma$ such that the restricted range $\Psi(\Omega \setminus A)$ is a separable set.
Pettis measurability theorem establishes that a vector valued function $\Psi$ is strongly measurable if and only if it is weakly measurable and essentially separably valued. In particular, if $E$ is separable we have that weak and strong measurability coincide. 

If $(\Omega, \Sigma,\mu)$ is a finite measure space, a Banach function space $X(\mu)$ is an ideal in $L^0(\mu)$ with the $\mu-$a.e. order that is a Banach space, satisfying that for  all $f,g \in X(\mu)$ such that $|f| \le |g|$ we have  $ \|f\|_{X(\mu)} \le \|g\|_{X(\mu)}.$ It is a Banach lattice, and we always assume that $\chi_{\Omega} \in X(\mu)$ (and so this function is an order unit).  We say that a Banach function space $X(\mu)$ has the Fatou property if for every increasing sequence $(f_n)_n \in X(\mu)$ that is norm bounded and converging $\mu-$a.e. to a measurable function $f,$ we have that $f \in X(\mu)$ too.

Recall now some notions and notation on K\"othe-Bochner spaces of vector valued  integrable functions. If $E$ is a Banach space and $X(\mu)$ is a Banach function space over the finite measure space $(\Omega, \Sigma, \mu),$ the K\"othe-Bochner space $X(\mu,E)$ is the Banach space of all the strongly measurable functions $\Psi: \Omega \to E$ such that the norm function $w \mapsto \| \Psi(w) \|_E $ belongs to $X(\mu).$ The norm in this space is given in this case by
$$
\| \Psi\|_{X(\mu,E)} = \Big\| \, \| \Psi(w) \|_E \, \Big\|_{X(\mu)}.
$$
Particular cases of these spaces are the Bochner spaces of $p-$integrable functions for $1 \le p <\infty,$ that are given when $X(\mu)=L^p(\mu).$
In this case, the space $L^p(\mu, E)$ can be represented as  the completion of the  tensor product $L^p(\mu) \widehat \otimes_{\Delta_p} E,$ where $\Delta_p$ is the norm for the tensor product inherited from this identification. 


The space $\mathcal M ( X(\mu),Y(\mu) )$ is the space of multiplication operators from $X(\mu)$ to $Y(\mu)$ with the usual operator norm (see \cite{caldel,mali}). We identify each of these operators with the function that defines it, although we sometimes use the notation $M_g$ for the operator defined by the function $g$ to distinguish it from the function $g$ itself.   If the measure $\mu$ is fixed in the context, we will write  $\mathcal M ( X,Y)$ for this space for simplicity.

We refer to \cite{cobzas} for main issues on Lipschitz operators, to \cite{lind,ok} for definitions and results on Banach function spaces, to \cite{defa,DieJarTon} for general issues on factorization in spaces of linear operators and normed tensor products, and to \cite{DiUhl} for (vector-valued) strongly measurable functions.

\section{Main reference: the linear case}

Let $(\Omega,\Sigma,\mu)$ be a finite measure space, and $X(\mu)$ and $Y(\mu)$ two Banach function spaces. Consider the space of measurable functions $\mathcal M(X(\mu), Y(\mu))$, that define multiplication operators $M_h:X(\mu) \to Y(\mu)$ given by  $M_h(f)(w)=h(w) \cdot f(w),$  $f \in X(\mu).$ Recall that, depending on $X(\mu)$ and $Y(\mu),$ that space could only contain the $0$ function. We will see that the adaptation of this notion to the Lipschitz case fits the concept of lattice Lipschitz operator, that can be more specified for the concrete case of operators between Banach function spaces.

\begin{definition}
	\label{def:lattice_lips}
	We say that a (non-linear)  operator between Banach function spaces $T:X(\mu) \to Y(\mu)$ is lattice Lipschitz if there exists a real-valued non-negative function $K(\cdot)$ such that for every pair of functions $f,g \in X(\mu)$ and $w \in \Omega,$
	$$
	\big| T(f)(w) - T(g)(w) \big| \le K(w) \, \big| f(w)- g(w) \big| \quad \mu-a.e. 
	$$
	In case $K(\cdot)$ is $\mu$-measurable, the condition can be expressed as $|T(f) - T(g)| \leq K |f-g|$ with the pointwise order on $L^0(\mu)$.
	Extending the classical terminology for constants appearing in Lipschitz inequalities, the function $K$ will be called a bound or associated function for $T$.
	We will call also lattice Lipschitz operators  to maps $T: X(\mu) \to L^0(\mu)$.
\end{definition}

%

We start by a basic property, which gives a relation between lattice Lipschitz operators and the linear ones.  A related result, but given in a different framework, can be found in \cite[Lemma 1.1]{appLib} and the comments after it.

\begin{proposition}
	\label{propo:llips_basic_X}
	Let $(\Omega, \Sigma, \mu)$ be a measure space and $X(\mu), Y(\mu)$ be two Banach function spaces (or, eventually, $Y(\mu) = L^0(\mu)$).
	Consider a lattice Lispchitz $T : X(\mu) \to Y(\mu)$ operator that satisfies that $T(0) = 0.$ 
	Then, for every $f,g \in X(\mu)$ and disjoint sets $A,B \in \Sigma$,
	\begin{itemize}
		\item[(i)] $T(f \chi_A) = T(f) \chi_A$ $\mu-$a.e., and so
		\item[(ii)]  $T(f \chi_A + g \chi_B)= T(f \chi_A) + T(g \chi_B)$ $\mu-$a.e.
	\end{itemize}
\end{proposition}

\begin{proof}
	Let $f$ and $A$ be as in the statement.
	Since $f \chi_A \in X(\mu)$, there exists a null set $N$ such that
	\begin{equation*}
		| T(f) - T(f \chi_A) |(w) \leq K(w) \cdot | f - f \chi_A |(w), \quad w \in \Omega \setminus N,
	\end{equation*}
	and another null set $N'$ that $ | T(f \chi_A) |(w) \leq K(w) | f \chi_A |(w) $
	for each $w \in \Omega \setminus N'$.
	Let $w \in \Omega \setminus (N \cup N')$.
	If $w \not\in A$,
	\begin{equation*}
		| T(f) \chi_A - T(f \chi_A) |(w) = | 0 - T(f \chi_A)(w) | \leq K(w) \cdot | f(w) | \chi_A (w) = 0.
	\end{equation*}
	If $w \in A$,
	\begin{align*}
		| T(f) \chi_A - T(f \chi_A) |(w)
		& = | T(f)(w) - T(f \chi_A)(w) | \\
		& \leq K(w) | f(w) - (f \chi_A) (w) | = 0.
	\end{align*}
	Then, $T(f) \chi_A = T(f \chi_A)$  $\mu-$a.e. This gives (i). For (ii), just note that using (i) we get $\mu-$a.e.
	\begin{align*}
		T(f \chi_A + g \chi_B)
		& = T \big( (f \chi_A + g \chi_B) (\chi_A + \chi_B) \big)
		= T(f \chi_A + g \chi_B) (\chi_A + \chi_B) \\
		& = T(f \chi_A +  g \chi_B) \chi_A+ T(f \chi_A +  g \chi_B) \chi_B \\
		& = T \big( (f \chi_A + g \chi_B) \chi_A \big) + T \big( (f \chi_A + g \chi_B) \chi_B \big)
		= T(f \chi_A) + T(g \chi_B).
	\end{align*}
\end{proof}

We start by giving the linear reference, that is, the characterization of the lattice Lipschitz maps for the simplest case when the map is linear.
Of course it is essentially known, but it will help to find what we can expect for the non-linear case. 

\begin{proposition}
	\label{proplin}
	Let $T:X(\mu) \to Y(\mu)$ be a {\bf linear} continuous map. Then it is a lattice Lipschitz operator if and only if there is a function $h \in \mathcal M(X(\mu),Y(\mu))$ such that $T=M_h.$
	In this case, $K(w) = |h(w)|$ is a bound function for $T$. 
\end{proposition}

\begin{proof}
	A well-known Radon-Nikodym Theorem argument proves the result; let us write the detailed proof for the aim of completeness, since some of the ideas in it will be used for the non-linear case. 
	
	Suppose first that $T$  is lattice Lipschitz.
	Assume without loss of generality that the measure $\mu$ is finite.
	Since $T$ is linear, $T(0)=0$ and we can use the lattice Lipschitz order inequality for single functions, and not necessarily for differences of functions. All the inequalities below are defined for concrete functions of the equivalence classes of the functions involved and has to be understood $\mu-$a.e., even if this is not explicitly written.   Consider the set function $\nu:\Sigma \to \mathbb R$ given by $A \mapsto \nu(A):=\int_\Omega T(\chi_A) d \mu.$ Note that $T(\chi_\Omega)  = h \in Y(\mu) \subset L^1(\mu),$ since $\mu$ is finite, and so $\nu(A) = \int_A T(\chi_\Omega)$ for all $A \in \Sigma$ by Proposition  \ref{propo:llips_basic_X}.

	
%
	
	1) 
	As $h$ is a $\mu$-measurable function, we assume that is non-negative.
	Otherwise consider it as a difference of two positive functions and repeat the reasoning for the difference of two integrals in the above equation.
	Let $ f \in X(\mu)$ be a simple function and take $B \in \Sigma.$
	Again by Proposition \ref{propo:llips_basic_X} and the linearity of $T$, we get
	\begin{equation*}
		\int_B T(f) d \mu = \int_\Omega T(f) \chi_B d \mu = \int_\Omega T(f \chi_B) d \mu= \int_\Omega h \, f \chi_B  \, d \mu= \int_B h \, f \, d \mu.
	\end{equation*}
	As a consequence, $T(f) = h \cdot f$
	
	2)
	Fix now $f \in X(\mu)$ and assume without loss of generality that $f \geq 0$.
	Take an increasing sequence of simple function such that $f_n \uparrow f$ pointwise and observe that, by the non negativity of $h$, $T(f_n) = h \cdot f_n \uparrow h \cdot f = T(f_n)$.
	Then, by the Monotone Convergence Theorem and the previous paragraph we get that 
	\begin{equation*}
		\int_B T(f) \, d\mu = \lim_n \int_B T(f_n) \, d\mu = \lim_n \int_B f_n \, d\nu = \int_B f \, d\nu,
	\end{equation*}
	for any $B \in \Sigma$. Then, $T(f) = h \cdot f = M_h(f)$
	Moreover, since it is defined for every $f \in X(\mu),$ we get that $h \in \mathcal M(X(\mu),Y(\mu)).$
	
	For the converse implication, given a multiplication operator $M_h \in \mathcal M(X(\mu), Y(\mu)),$ we only need to consider the following  direct  inequalities, for $f,g \in X(\mu);$
	\begin{equation*}
		\big| M_h(f)- M_h(g) \big|(w) = | M_h(f-g)|(w) = |h(w)(f-g)(w)| \le  |h(w) | \, \big| f(w)-g(w) \big|,
	\end{equation*}
	and so the result holds for $K(w)=|h(w)|$ $\mu-$a.e.

\end{proof}

\section{Pointwise diagonal Lipschitz operators}

Following the idea that suggests the case of multiplication operators, we establish in this section what would be the Lipschitz version of a multiplication operator. Let us
show the arguments to find it.
Before starting to explain our ideas, let us recall some fundamental definitions and results. 


We will find the fundamental functions to represent this class of operators as Bochner integrable functions.
Consider the vector space of all the real-valued functions of one real variable $\mathcal{RV}(\mathbb R).$
Let $\mathcal F(\mathbb R) \subset \mathcal{RV}(\mathbb R)$ be a Banach space and take a vector valued function $\Phi: \Omega \to \mathcal F(\mathbb R)$ (for example a strongly measurable function).
Let $Z$ be a Banach space of (classes of a.e. equal) real functions; we are mainly (but not only) thinking on $Z=Y(\mu),$ a Banach function space over $\mu.$
We want to show that we can define an operator $T: X(\mu) \to Z$ using $\Phi$ and when it satisfies Definition \ref{def:lattice_lips}.

\begin{example}
	\label{ex:first_const}
	Consider $\mathcal F(\mathbb R) = Lip(\mathbb R),$ the space of Lipschitz real functions of one real variable. Consider the constant function $\Phi: \Omega \to Lip(\mathbb R)$  given by
	$\Phi(w)(v):=\frac{1}{1+|v|}$ for all $w \in \Omega.$
	Let $X(\mu)$ be any Banach function space in a finite measure space and define $T: X(\mu) \to X(\mu)$ as
	$$
	T(f)(w)=  \big( \Phi(w) \circ f \big) (w) = \Phi(w) \big(  f(w) \big) = \frac{1}{1+|f(w)|},  \quad f \in X(\mu), \,\,\, w \in \Omega \,\,\, \mu-a.e.
	$$
	Note that, for all $f,g \in X(\mu)$ and $w \in \Omega,$
	\begin{align*}
		\big| T(f)(w) - T(g)(w) \big|
		& = \Big| \frac{1}{1+|f(w)|} - \frac{1}{1+|g(w)|} \Big|
		= \Big| \frac{|g(w)|- |f(w)|}{(1+|f(w)|)(1+|g(w)|)} \Big| \\
		& \le \frac{ \big| g(w)- f(w)  \big| }{(1+|f(w)|)(1+|g(w)|)}
		\le \big| g(w)- f(w)  \big|.
	\end{align*}
	Therefore, it is lattice Lipschitz, and $K(w)$ can be defined to be the constant function $\mathbf 1.$ The function $g$ in Remark \ref{Tsub} is in this case the constant function $\mathbf 1,$ since $T(0)= \mathbf 1.$
\end{example}

\begin{remark}
	Let us show first that this covers \textbf{the linear case} explained in Proposition \ref{proplin}.  If  $\mathcal F(\mathbb R) = L(\mathbb R) \subseteq Lip_0(\mathbb R),$ the space of linear real functions of one real variable, of course this space can be identified with the space of the real numbers $\mathbb R$ by means of the map $\mathbb R \ni r \mapsto f_r(\cdot)= r \cdot (\cdot) \in L(\mathbb R).$
	
	Take now a function $h \in \mathcal M(X(\mu),Y(\mu))$ and consider the Bochner integrable function
	\begin{align*}
		\Phi_h : \Omega & \to L(\mathbb R) \\
		w & \mapsto \Phi_h(w)(\cdot) = f_{h(w)} (\cdot)= h(w) \cdot (\cdot).
	\end{align*}
	That is, it takes at each $w$ the value provided by the multiplication of the real number $h(w)$ and the variable. Then the operator can be given by the composition
	\begin{equation*}
		T(f)(w) = \big( \Phi_h(w) \circ f\big)(w)
		= h(w) \cdot f(w) = M_h(f)(w),  \quad f \in X(\mu),
	\end{equation*}
	that is well-defined $\mu-$a.e. Recall that, by Proposition \ref{proplin}, the class of operators above that are also linear coincides with the class of linear lattice Lipschitz operators, that is, those satisfying that for all $f \in X(\mu), \ \big| T(f)(w) \big| \le K(w) \, |f(w)|,$ for a certain function $K(w)$ belonging to $\mathcal M(X(\mu),Y(\mu)).$
\end{remark}

Let us restrict our attention to vector valued functions for the case that $\mathcal F(\mathbb R) = Lip_0(\mathbb R)$ with its Lipschitz constant as norm.
The following results illustrate the construction of lattice Lipschitz operators from a vector-valued function, as in the cases above.
The conditions of the map $\Phi: \Omega \to Lip_0(\mathbb R)$ in terms of measurability needed for having a well defined lattice Lipschitz operator will be discussed hereafter.

Consider for the moment $\Phi: \Omega \to Lip_0(\mathbb R)$ a strongly mesurable function, as an almost everywhere limit of a sequence of simple functions defined by $\psi(w) = \sum_{i=1}^n \chi_{A_i}(w) \phi_i.$  As for the notation, and to facilitate the understanding of the arguments we will also write these functions as $\psi(w) = \sum_{i=1}^n  \phi_i \chi_{A_i}(w).$ 
It is relevant to mention that such a function $\psi$ defines an operator $T_\psi: X(\mu) \to Y(\mu)$ acting as a superposition operator, that is, by the natural formula
\begin{equation}
	\label{eq:simple_T_natural}
	T_\psi(f)(w)= \sum_{i=1}^n \chi_{A_i}(w) \phi_i \big( f(w) \big), \quad w \in \Omega
\end{equation}
whenever for each $i=1,...,n,$ the functions $w \mapsto \phi_i(f(w))$ belong to $Y(\mu).$ This formula has to be understood as defined $\mu-$a.e.; indeed, if $f$ and $g$ are measurable functions that are equal but in a $\mu-$null set $A \in \Sigma$ and its equivalence class belongs to $X(\mu)$, we clearly have that $T_\psi(f)(w)= T_\psi(g)(w)$ for $w \in \Omega \setminus A,$ and so $T_\psi(f) \in Y(\mu).$
In other words, the map $T_\psi$ is defined independently of the representative we take for the function $f.$


\begin{lemma}
	\label{simplef}
Using the notation above, 
	every simple function $\psi = \sum_{i=1}^n \phi_i \chi_{A_i} \in X(\mu, Lip_0(\mathbb R))$ gives a lattice Lipschitz operator
	$T_\psi: X(\mu) \to X(\mu)$ with bound function
	\begin{equation*}
		K_{T_\psi}=\sum_{i=1}^n \chi_{A_i}  Lip(\phi_i) \in L^\infty(\mu).
	\end{equation*}
\end{lemma}

\begin{proof}
	Note first that the composition of a measurable function $f$ with a Lipschitz (and so continuous) map $\phi \in Lip_0(\mathbb R)$ gives also a measurable function, and
	$$
	|\phi \circ f(w) | = | \phi \circ f(w)  - \phi \circ 0 (w)|  = |\phi(f(w))- \phi(0) | \le Lip(\phi) \, |f(w)|,
	$$
	where the inequality is defined $\mu-$a.e.,	and so $\phi \circ f \in X(\mu)$ too due to the ideal property of the Banach function spaces.
	
	Let us compute the value of the Lipschitz bound function for  the operator $T_\psi:X(\mu) \to X(\mu)$ associated to a simple function 
	$\psi = \sum_{i=1}^n \phi_i \chi_{A_i},$ $\phi_i \in Lip_0(\mathbb R).$
	For every $f , g \in X(\mu),$ we have
	\begin{align*}
		| T_\psi(f)(w) -T_\psi(g)(w) |
		& = \left| \sum_{i=1}^n \chi_{A_i}(w) \cdot \big( \phi_i (f(w))- \phi_i (g(w)) \big) \right| \\
		& \le \sum_{i=1}^n \chi_{A_i}(w) \big| \phi_i (f(w))- \phi_i (g(w)) \big| \\
		& \le \sum_{i=1}^n \chi_{A_i}(w) Lip(\phi_i) \big| f(w)- g(w) \big| \\
		& =  \left( \sum_{i=1}^n \chi_{A_i}(w) Lip(\phi_i) \right) \cdot \big| f(w)- g(w) \big| .
	\end{align*}
	
\end{proof}

Note that the next result gives a constructive procedure for finding well-defined superposition operators as limits of simple functions. This is not the usual way of introducing such superposition maps (\cite{appel,appel2}), and opens the door to a new point of view for understanding superposition operators as vector-valued integrable functions.

\begin{proposition}
	\label{lemma:rep}
	Let $X(\mu)$ and $Y(\mu)$ be  Banach function spaces and let  $\Phi:\Omega \to  Lip_0(\mathbb R)$ be a strongly measurable function.
	Then 
	\begin{itemize}
		\item[(i)] the composition (superposition)  formula
		\begin{equation*}
			T_\Phi(f)(w) = \big( \Phi(w) \circ f \big) (w), \quad w \in \Omega
		\end{equation*}
		gives a well-defined map $T_\Phi:X(\mu) \to L^0(\mu),$ and
		
		\item[(ii)] if $(\psi_n)_n$ is a sequence of simple functions converging $\mu-$a.e. to $\Phi$ with bound functions $K_{T_{\psi_n}}$ such that
		\begin{equation*}
			\sup_n K_{T_{\psi_n}} \in \mathcal M(X(\mu),Y(\mu)),
		\end{equation*}
		then $T_\Phi$ gives a lattice Lipschitz map $T_\Phi:X(\mu) \to Y(\mu).$
	\end{itemize}
\end{proposition}

\begin{proof}
	(i) We only need to prove that the composition with a certain function $f$ is well-defined as a representative of a $\mu-$a.e. class of functions of $Y(\mu).$ Since $\Phi$ is strongly measurable, we get that there is a $\mu-$null set $B$ and a sequence of simple functions $\psi_n$
	such that  for every $w \in \Omega \setminus B$
	$$
	\Phi(w) = \lim_n \psi_n(w).
	$$
	Since $\psi_n(w)$ are simple functions with values in $Lip_0(\mathbb R),$ by Lemma \ref{simplef} the maps $T_{\psi_n}: X(\mu) \to X(\mu) \subset L^0(\mu)$ are well-defined.  We have that for $ 0 \in X(\mu)$ and every $n,$ $\psi_n(w)(0)=0$ for all $w \in \Omega,$ and so $\Phi(w)(0)=0$ $\mu-$a.e.  Thus, for every $f \in X(\mu)$ and almost all $w \in \Omega,$ 
	\begin{align*}
		\big| \Phi(w)(f(w)) - \psi_n(w)(f(w))  \big|
		& =	\big| \Phi(w)(f(w)) - \psi_n(w)(f(w)) - \Phi(w)(0) - \psi_n(w)(0) \big| \\
		& \le Lip \big( \Phi(w) - \psi_n(w) \big) \, |f(w)-0| \to_n  0.
	\end{align*}
	Therefore, $\lim_n \psi_n(w)(f(w))= \Phi(w)(f(w))$ $\mu-$a.e., and so the operator $T_\Phi$ is defined at least pointwise for all $f \in X(\mu),$ giving a measurable function for each such $f$ that is an a.e. pointwise limit of functions in $X(\mu).$
	
	(ii)  For a.e.  every $w \in \Omega$ and $n \in \mathbb N,$
	\begin{align*}
		| \Phi(w) ( f(w)) |
		& \le \big| \Phi(w)(f(w)) - \psi_n(w)(f(w)) \big| + |\psi_n(w)(f(w)) | \\
		& \le Lip \big( \Phi(w) - \psi_n(w) \big) \, |f(w)| + \sup_n K_{T_{\psi_n}}(w) \cdot |f(w)|.
	\end{align*}
	Since $Lip \big( \Phi(w) - \psi_n(w) \big) \to_n 0$  and  $\sup_n K_{T_{\psi_n}}(w) \cdot |f(w)| \in Y(\mu)$ by hypothesis, we have that
	$| \Phi(w) ( f(w)) | $ is a measurable function a.e. bounded by a function in $Y(\mu).$ By the ideal property of the Banach function space $Y(\mu),$ we get that the operator is well-defined.
	
	On the other hand, if $f,g \in X(\mu)$ we have that for a.e. $w,$
	\begin{align*}
		| \Phi(w) & (f(w)) - \Phi(w)(g(w)) | \\
		& \le \big| \Phi(w)(f(w)) - \psi_n(w)(f(w)) - \Phi(w)(g(w)) + \psi_n(w)(g(w)) \big| \\
		& \qquad + | \psi_n(w)(f(w)) - \psi_n(w)(g(w)) | \\
		& \le Lip \big( \Phi(w) - \psi_n(w) \big) \cdot |f(w) - g(w)| + \sup_n K_{T_{\psi_n}}(w)  \cdot |f(w)-g(w)|, 
	\end{align*}
	and since this holds for every $n,$ we obtain
	\begin{equation*}
		| T_\Phi (f(w)) - T_\Phi (g(w)) | \le \sup_n K_{T_{\psi_n}}(w) \cdot |f(w)-g(w)|
	\end{equation*}
	and  so $T_\Phi$ is a lattice Lipschitz operator for which we can get a  bound function that defines a multiplication operator from $X(\mu)$ to $Y(\mu),$ since
	\begin{equation*}
		K_\Phi \le \sup_n K_{T_{\psi_n}} \in \mathcal M(X(\mu),Y(\mu)).
	\end{equation*}	
\end{proof}

The operator $T(f)(w) = \Phi(w)(f(w))$ used in Example \ref{ex:first_const} and Proposition \ref{lemma:rep} will be called superposition operator  an will be denoted by $T_\Phi(f) = \Phi \circ  f$ by abuse of notation.


%

\vspace{0.3cm}

It is relevant to mention that we can define lattice Lipschitz operators for which it is not possible to provide  a representation by means of a vector-valued function $\Phi.$ We show an example of this situation below.

\begin{example}
Consider the Lebesgue measure space $([0,1],\Sigma,\mu)$ and 
take a pointwise bounded non-measurable function $h: [0,1] \to \mathbb R$ that is positive and with pointwise bound $K>0$ (for example, the function given by 1 plus the characteristic function of a non-measurable set). Consider the space of essentially bounded measurable functions $X(\mu) = L^\infty(\mu)$ and consider the multiplication operator $M_h: X(\mu) \to \mathcal F(\mathbb R),$ where  $M_h(f)= h \cdot f,$ and $\mathcal F(\mathbb R)$ is the space of all classes of real functions with real variable that are equal $\mu-$a.e. Note that this map is well-defined when defined from a Banach function space $X(\mu)$ over $\mu.$

This means that we can consider $M_h$ defined from $X(\mu)$ to  the space $\mathcal F(\mathbb R)_\mu$ of all the equivalence classes of functions that are equal $\mu-$a.e.  A Banach lattice structure can be given for a subspace $Z$ of this space by dividing every function in it by $h;$ for example, the (classes of) functions that satisfy  that $ \mathcal F(\mathbb R)_\mu \ni v \mapsto v/h \in X(\mu)$ defines a Banach lattice with $\mu-$a.e. order and the norm given by this map.  $M_h:X(\mu) \to Z$  is clearly well-defined.
On the other hand, for every $w$ and every pair of measurable functions $f,g,$ 
\begin{equation*}
	| M_h(f)(w) - M_h(g)(w)|
	=  h(w)  \cdot |f(w)- g(w) | 	\le K \, |f(w)- g(w)|
\end{equation*}
and so it is  a lattice Lipschitz map. The vector valued function $\Phi$  that represents the lattice Lipschitz operator is $w \mapsto h(w)(\cdot),$ where the real number $h(w)$ is considered as a constant that defines a Lipschitz map in $Lip_0(\mathbb R),$ which cannot be, by construction, measurable: in other case, the scalar function  $w \mapsto \|\Phi(w)\|_{Lip(\mathbb R)} = h(w)$ would be measurable.
\end{example}

Let us explain an example in which a non-measurable Lipschitz valued function is constructed using a typical cardinality argument.
In order to do it, let us recall that the spaces of Lipschitz functions ares not separable in general, as a consequence of the fact that $\mathbb R$ is not totally bounded. 

\begin{remark}
	\label{remsep}
	
	The space of Lipschitz bounded funrction $(Lip_b(\mathbb R), \| \cdot\|_\infty)$ and the space of Lipschitz functions that vanish at 0 $(Lip_0(\mathbb R), Lip(\cdot))$ are not separable.
	To see this, consider the set of natural numbers $\mathbb N$ as a subset of $\mathbb R.$ For every subset $S \subset \mathbb N \setminus \{0\},$ define the function $f_S: \mathbb R \to \mathbb R$ given by 
	\begin{equation*}
		r \mapsto \min \Big\{\frac{1}{4}, \min\{|r-s|: s \in S \cup \{0\} \} \Big\} = \min \Big\{\frac{1}{4}, |r|, d(r, S) \} \Big\},
	\end{equation*}
	which are $1$-Lispchitz functions bounded by $1/4$.
	
	It is immediate to notice that $\| f_S - f_D \|_\infty \ge 1/4,$ which implies that these functions are separated, what, 
	together with the fact that there are as many functions as subsets of $\mathbb N \setminus \{0\},$ gives that the family $\mathcal H:=\{ f_S: S \subset \mathbb N\}$ cannot the 
	approximated by any countable family of functions. So, $(Lip_b(\mathbb R), \| \cdot\|_\infty)$ is not separable. 
	
	Note also that the same construction can be used to prove that $(Lip_0(\mathbb R), Lip(\cdot) )$ is not separable. Indeed, for all the functions $f_S$ we have that $f_S(0)=0$ and $Lip(f_S-f_D) = 1$ for all $S,D \subset \mathbb N \setminus 
	\{0\}$ such that $S \neq D.$
	
\end{remark}

Let us show some examples of non-strongly measurable functions that still define lattice Lipschitz operators. We will show examples in (non $\sigma-$finite) discrete measure spaces as well as in finite measure spaces (Lebesgue measure). 

\begin{example} 
	
	Consider  the measure space $(B_{Lip_b(\mathbb R)},\mathcal P (B_{Lip_b(\mathbb R)}), c)$, where $c$ is the counting (extended) measure.
	Since $Lip_b(\mathbb R)$ is not separable, take a non-countable set $N$ in it that cannot be approximated by any countable set (for example, as in Remark \ref{remsep}).
	
	Define the vector valued function $\Phi: B_{Lip_b(\mathbb R)} \to B_{Lip_b(\mathbb R)}$ by $\Phi = id \cdot \chi_N$
	that is,
	\begin{equation*}
		\Phi (\tau) = \sum_{h \in N} h \cdot \chi_{\{h\}} (\tau) =  \left\{
		\begin{array}{ll}
			\tau     & \mathrm{if\ }  \tau  \in N \\
			0 & \mathrm{if\ }  \tau \notin N \, .
		\end{array}
		\right. 
	\end{equation*}
	It cannot be strongly measurable, since if there is a sequence $(s_n)_{n=1}^\infty$ of simple functions such that $s_n = \sum_{i=1}^{m_n} \eta_n^i \chi_{A_n^i} \to \Phi$ pointwise, then for all $\tau \in N$, $s_n(\tau) = \eta_n^{i_n} \to \Phi(\tau) = \tau$.
	
	Note, however, that we can define a lattice Lipschitz operator 
	\begin{equation*}
		T_\Phi:  \ell^\infty(B_{Lip_b(\mathbb R)} )  \to \ell^\infty( B_{Lip_b(\mathbb R)} )
	\end{equation*}
	by pointwise composition, as we have done in the representation lemma \ref{lemma:rep}. Indeed, 
	\begin{equation*}
		T_\Phi \big( (r_\tau)_{\tau \in B_{Lip_b(\mathbb R)}} \big) = (\Phi(\tau)(r_\tau))_{\tau \in B_{Lip_b(\mathbb R)}},
	\end{equation*}
	Therefore, for any $r, t \in \ell^\infty(B_{Lip_b(\mathbb R)} )$ and $\eta \in B_{Lip_b(\mathbb R)}$,
	\begin{equation*}
		\Big| T_\Phi(r)(\eta) - T_\Phi(t)(\eta) \Big|
		=   \Big| \Phi(\eta)(r(\eta)) - \Phi(\eta)(t(\eta)) \Big| 
		\le Lip(\Phi(\eta)) \big| r(\eta)- t(\eta) \big|.
	\end{equation*}
	Thus, $T_\Phi$ is lattice Lipschitz, with bound function $K = Lip(\Phi(\cdot)) = \chi_{N}.$
	Note that the same construction can be done for $Lip_0(\mathbb R)$ instead of $Lip_b(\mathbb R).$
	
	This example allows a general version, taking into account Pettis measurability theorem. Note first that the only condition required in the
	example above with the counting measure is the existence of a bijection relating any element of the measure space with an element of the set $\mathcal H$ defined by the functions $f_S.$  If $(\Omega, \Sigma,\mu)$ is a measure space, any function $\Phi$  that satisfies that $\Phi(\Omega \setminus A) \cap  \mathcal H$ is uncountable for any $\mu-$null set $A \in \Sigma$ gives a similar example. For instance, consider the measure space $( \mathcal P(\mathbb N \setminus \{0\}), \mathcal P \big(  \mathcal P(\mathbb N\setminus \{0\}) \big), \nu),$ where $\nu$ is given by
	$$
	\nu (\{ S\})= \sum_{n \in S} \frac{1}{2^n}, \quad \text{and} \quad \nu(\Lambda)= \sum_{S \in \Lambda} \nu(\{S\}).
	$$
	The function $\Phi:  \mathcal P(\mathbb N \setminus \{0\}) \to Lip_0(\mathbb R)$ given by $\Phi(S)= f_S$ is well-defined, and obviously non-strongly measurable, due to the non-countable nature of  $ \mathcal P(\mathbb N \setminus \{0\}).$
\end{example}

Next results shows where it could be ensured that the range of $T$ is a space of measurable functions in therms of $\Phi$.

\begin{lemma}
	\label{lem:def_Phi}
	Let $(\Omega, \Sigma, \mu)$ be a measure space  and $T: X(\mu) \to L^0(\mu)$ a lattice Lipschitz operator.
	Then, there exists $\Phi: \mathbb R \times \Omega \to \mathbb R$ such that
	\begin{equation*}
		\Phi(\cdot)(\lambda) = T(\lambda \chi_\Omega), \quad \lambda \in \mathbb R,
	\end{equation*}
	(that is, $\Phi(\cdot)(\lambda)$ is a particular representative of $T(\lambda \chi_\Omega)$),
	and for any $w \in \Omega,$ $\Phi(w): \mathbb R \to \mathbb R$ is a $K(w)$-Lipschitz function.
\end{lemma}

\begin{proof}
	Denote by $\mathfrak M$ the set of all real $\mu$-measurable functions defined in all $\Omega$ without the usual $\mu$-almost everywhere equivalence.
	That is, all possible   representatives of classes of functions in $L^0(\mu)$.
	We claim that there exists $\Psi: \mathbb Q \to \mathfrak M$ such that $\Psi(\lambda)$ is a representative of $T(\lambda \chi_\Omega)$ defined in all $\Omega$ and for any $\lambda, \xi \in \mathbb Q$
	\begin{equation}
		\label{eq:Phi_Lips_Q}
		| \Psi(\lambda)(w) - \Psi(\xi)(w) |
		\leq 
		K(w) | \lambda - \xi | \quad \text{for all} \quad w \in \Omega.
	\end{equation}
	That is, for any $w \in \Omega$, $\Psi(\cdot)(w): \mathbb Q \to \mathbb R$ is a $K(w)$-Lipschitz function.
	
	Let $\mathbb Q = ( \lambda_n )_{n=1}^\infty$, and choose $\Psi(\lambda_1)$ to be any representative of $T(\lambda_1 \chi_\Omega)$ defined in all $\Omega$.
	Assume now that $\Psi$ is defined in $\{ \lambda_n \}_{n=1}^{m-1}$ satisfying the given property.
	Choose $\sigma: \Omega \to \mathbb R$ to be any representative of $T(\lambda_m \chi_\Omega)$.
	By hypothesis, there exists $N_m^n$ null sets such that
	\begin{equation*}
		| \Psi(\lambda_n)(w) - \sigma(w) |
		\leq | \lambda_n \chi_\Omega(w) - \lambda_m \chi_\Omega(w) |
		= K(w) | \lambda_n - \lambda_m |
	\end{equation*}
	for any $w \not\in N_m^n$, $n \leq m$. Let $N_m = N_m^1 \cup N_m^2 \cup \cdots \cup N_m^{n-1}$ and define $\Psi(\lambda_m): \Omega \to \mathbb R$ as
	\begin{equation*}
		\Psi(\lambda_m)(w) =
		\begin{cases}
			\sigma(w) & w \not\in N_m \\
			\Psi(\lambda_1)(w) & w \in N_m^1 \\
			\Psi(\lambda_2)(w) & w \in N_m^2 \setminus N_m^1 \\
			\quad \vdots \\
			\Psi(\lambda_{m-1})(w) & w \in N_m^{m-1} \setminus (N_m^1 \cup \cdots \cup N_m^{m-2}) \\		
		\end{cases}.
	\end{equation*}
	Clearly, \eqref{eq:Phi_Lips_Q} is satisfied for any $\lambda_m, \lambda_n$ in all $\Omega$, also for elements in $N_m$.
	Proceeding inductively, $\Psi$ is correctly defined in $\mathbb Q$.
	
	Define $\Phi(w)(\lambda) = \Psi(\lambda)(w)$ in $\Omega \times \mathbb Q$.
	Then, given any $w \in \Omega$, $\Phi(w)\!\!\restriction_\mathbb Q$ is a $K(w)$-Lipschitz function, so it can be extended to $\mathbb R$ keeping the Lipschitz constant, for example following McShane formula, see \cite{mcs}.
	For $\lambda \in \mathbb R \setminus \mathbb Q$, consider $\tilde T(\lambda \chi_\Omega)(w)$ a representative of $T(\lambda \chi_\Omega)(w)$.
	By the continuity of $\Phi(w)$, if $\mathbb Q \ni \lambda_n \to \lambda$,
	\begin{align*}
		| \tilde T(\lambda \chi_\Omega)(w) - \Phi(w)(\lambda) |
		& = \lim_{n \to \infty} | \tilde T(\lambda \chi_\Omega)(w) - \Phi(w)(\lambda_n) | \\
		& \leq \lim_{n \to \infty} K(w) | \lambda - \lambda_n | = 0,
	\end{align*}
	for almost every $w \in \Omega$.
	Then, $\Phi(\cdot)(\lambda) : \Omega \to \mathbb R$ belongs to $\mathfrak M$ and is a representative of $T(\lambda \chi_\Omega)$.
\end{proof}

Let us recall the construction of a predual of $Lip_0(M)$ for any pointed metric space $M$.
For each $x \in M$, consider the evaluation functional $\delta_x: Lip_0(M) \to \mathbb R$ defined as $\delta_x(f) = f(x)$ as an element in $\big( Lip_0(M) \big)^*$.
Define the so called Lipschitz-free space $\mathfrak F(M) = \overline{\text{span}} \{ \delta_x : \ x \in M \} \subseteq \big( Lip_0(M) \big)^*$ using the dual norm which, in fact, satisfies $\| \delta_x - \delta_y \| = d(x,y)$ for $x, y \in M$.
Then, $\big( \mathfrak F(M) \big)^* = Lip_0(M)$.
This construction is explained in detail in \cite{cobzas}, where other constructions of a predual of $Lip_0(M)$, such as the Arens-Eells space $\AE(M)$, are also studied.

\begin{theorem}
	\label{theo:carac_Phi}
	Let $(\Omega, \Sigma, \mu)$ a measure space and $X(\mu)$ a Banach function space.
	The operator $T: X(\mu) \to L^0(\mu)$ is lattice Lipschitz if and only if there exists $\Phi: \Omega \to Lip_0(\mathbb R)$ such that $T(f)(w) = \Phi(w)(f(w))$ (for almost any $w \in \Omega$) and $p \circ \Phi: \Omega \to \mathbb R$ is $\mu$-measurable for any $p \in \mathfrak F(\mathbb R)$.
\end{theorem}

\begin{proof}
	Assume that $T$ is lattice Lispchitz.
	Define $\Phi(w)(\lambda)$ as the concrete representative of $T(\lambda \chi_\Omega)$ evaluated on $w$ given by Lemma \ref{lem:def_Phi}.
	For any $\lambda \in \mathbb R$, $\delta_\lambda \circ \Phi = \Phi(\cdot)(\lambda)$ is measurable since is a representative of $T(\lambda \chi_\Omega)$.
	By taking linear combinations and limits of function in $L^0(\mu)$, we obtain that $p \circ \Phi$ is $\mu$-measurable for any $p \in \mathfrak F(\mathbb R)$.
	
	For the reciprocal, let us see first that $T(f)(w) = \Phi(w)(f(w))$ is well defined from $X(\mu)$ to $L^0(\mu)$.
	Let $f = \sum_{i=1}^n \lambda_i \chi_{A_i} \in X(\mu)$ a simple function. Then,
	\begin{equation*}
		T(f)(w) = \Phi(w) \left( \sum_{i=1}^n \lambda_i \chi_{A_i}(w) \right)
		= \sum_{i=1}^n \Phi(w)(\lambda_i) \chi_{A_i}(w) \in L^0(\mu),
	\end{equation*}
	since $w \mapsto \Phi(w)(\lambda_i) = \langle \Phi(w), \delta_{\lambda_i} \rangle$ is $\mu$-measurable for all $i$.


Since every limit of simple functions is measurable, we have that $T(f) \in L^0(\mu)$ for all $f \in X(\mu).$
	Moreover, if $f = g$ $\mu$-a.e. in $\Omega$, clearly $T(f)(w) = \Phi(w)(f(w)) = \Phi(w)(g(w)) = T(f)(w)$ $\mu$-a.e., so $T$ is well defined.
	To see that $T$ satisfies the lattice Lipschitz inequality, let $f, g \in X$.
	For any choose of the representatives,
	\begin{equation*}
		| T(f)(w) - T(g)(w) | = | \Phi(w)(f(w)) - \Phi(w)(g(w)) | \leq Lip \ \Phi(w) \cdot | f(w) - g(w) |,
	\end{equation*}
	so we can select $K(w) = Lip \ \Phi(w)$.
\end{proof}

\begin{corollary}
	\label{coromeasu}
	If $\Phi: \Omega \to Lip_0(\mathbb R)$ is weakly $\mu$-measurable (or strongly $\mu$-measurable) and $X(\mu)$ is any Banach function space, then $T: X \to L^0(\mu)$ defined as $T(f)(w) = \Phi(w)(f(w))$ is a lattice Lipschitz operator.
\end{corollary}

\begin{proof}
	Assume that $\Phi$ is weakly measurable (since strong measurability implies it).
	Given any $p \in \mathfrak F(\mathbb R)$, as $p \in \big( \mathfrak F(\mathbb R) \big)^{**} = Lip_0(\mathbb R)^*$, $p \circ \Phi$ is $\mu$-measurable, so by Theorem \ref{theo:carac_Phi}, $T$ is a lattice Lipschitz operator from $X$ to $L^0(\mu)$.	
\end{proof}

\begin{example}
	Consider the Lebesgue measurable space $([0,1], \Sigma, \mu)$.
	Each real number in $[0,1]$ can be identified with a subset of natural numbers via a binary representation of its digits.
	Since that representation is not always unique, choose the one that is eventually 0 rather that 1.
	Define $\Phi$ as
	\begin{align*}
		\Phi: [0,1] & \to Lip_0(\mathbb R) \\
		w = \sum_{i=1}^\infty c_i \frac{1}{2^i} & \mapsto \left( \lambda \mapsto \min \left\{ \frac{1}{2}, |\lambda|, d(\lambda, \{i \in \mathbb N : c_i = 0 \}) \right\}  \right).
	\end{align*}
	
	Let us see that $T: L^1(\mu) \to L^1(\mu)$ defined by $T(f)(w) = \Phi(w)(f(w))$ is a well defined lattice Lispchitz operator using Theorem \ref{theo:carac_Phi}.
	Let $\lambda \in \mathbb R$ and consider $\delta_\lambda \in \mathfrak F(\mathbb R)$, we claim that $w \mapsto \langle \Phi(w), \delta_\lambda \rangle = \Phi(w)(\lambda)$ is $\mu$-measurable.
	Consider, for simplicity, the function
	\begin{equation*}
		\varphi_\lambda(w) = \min \left( \left\{ \frac{1}{2} \right\} \cup \big\{ | \lambda - i | : \ c_i = 0, \ i \in \mathbb N \big\} \right),
	\end{equation*}
	where $w = \sum_{i=1}^\infty c_i \frac{1}{2^i}$ is the binary representation of $w$ as before.
	Observe that $\Phi(w)(\lambda) = \min\{ |\lambda|, \varphi_\lambda(w) \}$, so it is enough to see that $\varphi_\lambda: \Omega \to [0,\frac{1}{2}] \subseteq \mathbb R$ is $\mu$-measurable.
	Let $0 < s < \frac{1}{2}$, and consider
	\begin{equation*}
		\varphi_\lambda^{-1} \big( ]-\infty,s[ \big) = \left\{ w = \sum_{i=1}^\infty c_i \frac{1}{2^i} \in \Omega : \ \{0,1\} \, \reflectbox{$\in$} \, c_i \not\to 1, \ \exists c_{i_0} = 0 : | \lambda - i_0 | < s \right\}.
	\end{equation*}
	If $d(\lambda, \mathbb N) \geq s$, $\varphi_\lambda^{-1} \big( ]-\infty,s[ \big) = \emptyset \in \Sigma$.
	Alternatively, pick $i_0$ to be the nearest natural number to $\lambda$, so
	\begin{align*}
		\varphi_\lambda^{-1} \big( ]-\infty,s[ \big) 
		& = \left\{ w = \sum_{i=1}^\infty c_i \dfrac{1}{2^i} \in \Omega : \ \{0,1\} \, \reflectbox{$\in$} \, c_i \not\to 1, \ c_{i_0} = 0 \right\} \\
		& = \bigcup_{j=0}^{2^{i_0-1}-1} \left[ \frac{2j}{2^{i_0}}, \frac{2j+1}{2^{i_0}} \right[ \in \Sigma,
	\end{align*} 
	which proves that $\varphi_\lambda$ is a Lebesgue measurable function.
	This situation, together with the fact that $Lip \ \Phi(w)$ is bounded by 1, shows that $T$ maps $L^1(\mu)$ to $L^1(\mu)$.
		
	If $w, \tau \in [0,1]$ are different, $Lip( \Phi(w) - \Phi(\tau) ) = 1$ (reasoning as in Remark \ref{remsep}).
	Then, $\Phi(\Omega)$ is not separable, so $\Phi$ cannot be approximated by simple functions.
	
	
	Observe that any Banach function space $X(\mu)$ can be considered instead of $L^1(\mu)$, having the same result.
\end{example}

\section{Lattice Lipschitz operators that are representable by strongly measurable functions}

We now show a representation theorem for lattice Lipschitz operators on spaces of Banach functions in terms of a vector-valued measurable function with values in $Lip_0(\mathbb R). $
We will see that these operators can in general be written in terms of pointwise diagonal Lipschitz maps, and sometimes a representation using a strongly measurable function is also available. The main technical problem throughout this section is the handling of the $\mu-$null sets appearing in Lipschitz-type inequalities, which have to be properly understood: a $\mu-$a.e. property sometimes concerns the terms involved in different ways. Essentially, these inequalities concern both the real numbers and the elements of $\Omega,$ and exchanging the roles of these variables is a delicate move.

Let us first prove the more general result, which does not explicitly require any measurability property of the fundamental function $\Phi_T.$ Instead, the function is assumed to act only on a special class of subsets of functions of $X(\mu),$ which are those of the type $f=\lambda \chi_\Omega$ for $\lambda \in \mathbb R.$We show that the lattice Lipschitz operators are disjoint, imitating the linear case and focusing on their ``diagonal'' nature.

\begin{proposition}
	\label{th1strong}
	Let $\Phi: \Omega \to Lip_0(\mathbb R)$ be a strongly measurable function such that $\Phi \in \mathcal M(X,Y)(\mu,Lip_0(\mathbb R))$.
	Consider the corresponding operator $T_\Phi: X(\mu) \to L_0(\mu)$, given by the pointwise composition $T_\Phi(f)(w) = \Phi(w)(f(w))$.
	Then, $T_\Phi$ is
	\begin{itemize}
		\item[i)] a well-defined operator $T_\Phi: X(\mu) \to Y(\mu)$,
		\item[ii)] Lipschitz continous, and
		\item[iii)] lattice Lipschitz with the best bound function appearing in the lattice Lipschitz inequality $K(\cdot) = Lip(\Phi(\cdot));$ that is, if there is any other measurable function $h$ such that the lattice Lipschitz inequality holds for $h(\cdot)= K(\cdot),$ then $Lip(\Phi(\cdot)) \le h(\cdot)$ $\mu-$a.e. 
	\end{itemize}
	
	Moreover, if there is another function $\Psi: \Omega \to Lip_0(\mathbb R)$ such that $\Phi=\Psi$ $\mu-$a.e., then $T_\Phi = T_\Psi$ and the bound functions $w \mapsto Lip(\Phi(w))$ and $w \mapsto Lip(\Psi(w))$
	are  $\mu-$a.e. equal. 
\end{proposition}

\begin{proof}
	Given $f \in X(\mu)$, by Corollary \ref{coromeasu}, $T(f)$ is a $\mu$-measurable function.
	As $\Phi(\Omega) \subseteq Lip_0(\mathbb R)$, $T(0) = 0$ and so for almost any $w \in \Omega$, $| T(f)(w) | \leq Lip(\Phi(w)) \cdot | f(w) |$.
	Then, since $Lip(\Phi(\cdot)) \in \mathcal M(X(\mu), Y(\mu))$, $Lip(\Phi(\cdot)) \cdot |f(w)| \in Y(\mu)$ and, by the ideal property of $Y(\mu)$, $T(f) \in Y(\mu)$.
	Moreover, given $f, g \in X(\mu)$
	\begin{equation*}
		\| T(f) - T(g) \|_{Y(\mu)} \leq \| Lip(\Phi(\cdot)) \|_{\mathcal M(X(\mu), Y(\mu))} \cdot \| f - g \|_{X(\mu)}.
	\end{equation*}	
	
	Consider now a pair of functions $f, g \in X(\mu)$ and $w \in \Omega$,
	\begin{equation*}
		| T(f)(w) - T(g)(w) |
		= | \Phi(w)( f(w))- \Phi(w)( g(w)) |
		\le Lip(\Phi(w)) \, | f(w)- g(w)|
	\end{equation*}
	and so $T_\Phi$ is lattice Lipschitz with bound function $K(w) = Lip(\Phi(w))$. Note that $\Phi(w)$ (and so  $Lip(\Phi(w))$) is defined $\mu-$a.e.

We \textit{claim} now that for every $w \in \Omega$ and every $\varepsilon >0$ there are functions $f_1,f_2 \in X(\mu)$ such that 
$$
	|T_\Phi(f_1)(w)-T_\Phi(f_2)(w)| >= (Lip(\Phi(w)) - \varepsilon)  \, {|f_1-f_2|(w)}.
$$	
Note that this implies that $Lip(\Phi(w))$ is the best constant that can be written in this inequality.
	Indeed, fix $w \in \Omega.$ Consider $\Phi(w) = \varphi_w \in Lip_0(\mathbb R).$  By definition of the Lipschitz constant, there are two different real numbers $r_1$ and $r_2$ such that   
	$$
	\frac{|\varphi_w(r_1)-\varphi_w(r_2)|}{|r_1-r_2|} > Lip(\varphi_w) - \varepsilon.
	$$
	Consider  the measurable  functions $\hat f_1=r_1 \chi_{\Omega} $ and $\hat f_2=r_2 \chi_{\Omega}.$  Then
	$$
	| (\Phi \circ \hat f_1)(w)- (\Phi \circ \hat f_2)(w) | = 
	| \varphi_w( r_1)- \varphi_w(r_2) |
	>  (Lip(\varphi_w) - \varepsilon) \, {|r_1-r_2|} .
	$$
Thus, since this can be done for all $r_1,r_2 \in \mathbb R$  for each $\varepsilon >0$ and  $w$ there are functions $f_1, f_2 \in X(\mu)$ having as representatives the  defined above  functions $\hat f_1$ and $\hat f_2$ in such a way that 
	$$
	|T_\Phi(f_1)(w)-T_\Phi(f_2)(w)| > (Lip(\Phi(w)) - \varepsilon)  \, {|r_1-r_2|}= (Lip(\Phi(w)) - \varepsilon)  \, {|f_1-f_2|(w)}.
	$$
Since this holds for all $\varepsilon,$ this means that the contant $Lip(\Phi(w))$ cannot be improved for this fixed $w.$  

Suppose now that there is another function $h \le Lip(\Phi(\cdot))$  such that  there is a non-null set $A$ satisfying
$ h \, \chi_A < Lip(\Phi(\cdot)) \, \chi_A$
and for all $f_1,f_2 \in X(\mu),$
$$
	|\Phi(f_1)(w)-T_\Phi(f_2)(w)|  \le  h(w) \, {|f_1(w)-f_2(w)|}, \quad w \in A.
$$
By the claim, this means that for every $r_1,r_2 \in \mathbb R, $  
$$
|\varphi_w(r_1)-\varphi_w(r_2)| \le h(w) | r_1-r_2| < Lip(\Phi(w))|r_1-r_2|,
$$
 for all $w$ belonging to a non-null set, what contradicts the definition of $Lip(\Phi(w)).$




	Finally, suppose that the measurable functions  $\Phi$ and $\Psi$ on $Lip(\mathbb R)$ are equal $\mu-$a.e. and let $f \in X(\mu).$ Clearly, for two representatives $\hat f_1$ and $\hat f_2$ of $f$ we get $\Phi \circ \hat f_1 =  \Psi \circ \hat f_2$ $\mu-$a.e. and so we clearly have that $T_\Phi$ and $T_\Psi$ are equal as operators from 
	$X(\mu)$ to $Y(\mu).$ The associated bound functions are equal $\mu-$a.e. as a consequence of the definition of the K\"othe-Bochner spaces: the measurable functions  $w \mapsto \| \Phi(w)\|_{Lip_0(\mathbb R) } =Lip(\Phi(w))$ and
	$w \mapsto \| \Psi(w)\|_{Lip_0(\mathbb R)} =Lip(\Psi(w))$ are equal $\mu-$a.e.
\end{proof}

Observe that, in contrast to this previous case, a general lattice Lipschitz operator can be non-Lipschitz.
For example, let $([0,1], \Sigma, \mu)$ be the Legesgue measure space and consider $X(\mu) = Y(\mu) = L^1(\mu)$, so $\mathcal M(X(\mu), Y(\mu)) = L^\infty(\mu)$.
Define $T: L^1(\mu) \to L^1(\mu)$ as $T(f) = \inf \{ f^2, 1/\sqrt{\cdot} \}$.
Then, $| T(f)(w) - T(g)(w) | \leq 2 / \sqrt{w} \cdot | f(w) - g(w) |$, but considering $f_n = n \chi_{[0, 1/n^4]}$ and the zero function, it can be shown that $T$ is not Lipschitz continous.
	
Moreover, if $T: X(\mu) \to Y(\mu)$ is a lattice Lipschitz operator such that $T(f)(w) = \Phi(w) \big( f(w) \big)$ for a strongly measurable function $\Phi : \Omega \to Lip_0(\mathbb R)$, there exists a sequence of simple vector-valued functions $\Phi_n : \Omega \to Lip_0(\mathbb R)$ such that $Lip(\Phi(w) - \Phi_n(w)) \to 0$ for almost any $w \in \Omega$.
But that convergence has no relation with the norm in $Y(\mu)$, so it should not be possible to conclude strong claims about the continuity of $T$ without more hypothesis, as in the previous proposition.
In fact, in more ``pathological" cases as the ones provided by $L^\infty(\mu)$, easy examples of lattice Lispchitz non continuous operators can be found.
For example, $T: L^1(\mu) \to L^\infty(\mu)$ defined as $T(f) = \sup \big\{ \inf \{ f, 1 \}, -1 \big\}$.
There, $\Phi$ is a constant function, so it is strongly measurable.
This example does not contradict Theorem 1 of \cite{appel}, since that result requires, in our context, $Lip \ \Phi(w)$ to belong to $\mathcal M(X(\mu), Y(\mu))$ which is only composed of the null function in this example.

The result above motivates the following definition, that will provide some representation results for the class of lattice Lipschitz operators with representing functions $\Phi$ satisfying certain properties.

\begin{definition}   \label{defSSL}
	We say that a map $T:X(\mu) \to Y(\mu)$ is  a strongly lattice Lipschitz operator if it can be written  as a pointwise composition $T(f)= \Phi \circ f$  for all $f \in X(\mu)$ with a strongly measurable function $\Phi: \Omega \to Lip_0(\mathbb R)$ belonging to the K\"othe-Bochner space $\mathcal M(X,Y)(\mu,Lip_0(\mathbb R)).$  We will write $SLL(X(\mu),Y(\mu))$ for the space of all such operators endowed with the norm 
	$$
	\| T_\Phi\|_{SLL(X(\mu),Y(\mu))} := \big\| Lip( \Phi(\cdot) )  \big\|_{\mathcal M(X,Y)}.
	$$
	Note that, in particular, $T(0)=0.$
\end{definition}

For example, given any Banach function space $X(\mu)$ since  $\mathcal M(X(\mu),X(\mu)) = L^\infty(\mu)$ (see \cite[Th.1]{mali})  we have that  $SLL(X(\mu), X(\mu))$ is the set of $T_\Phi : X(\mu) \to X(\mu)$ for which $\Phi : \Omega \to Lip_0(\mathbb R)$ is a strongly measurable function such that $\text{ess} \sup Lip(\Phi(\cdot))$ is finite.

Next result gives a characterization of strongly lattice Lipschitz operators.
We will need the requirement that $X(\mu) \subseteq Y(\mu);$ 
it is satisfied each time $\mu$ a finite measure and $\chi_\Omega \in \mathcal M(X,Y)$.
For example, let $\mu$ be the Lebesgue measure in $[0,1]$, $X(\mu) = L^p[0,1]$ and $Y(\mu) = L^q[0,1]$ for $1 \le p, q \le \infty.$
If $p < q$, as $L^p[0,1] \supset L^q[0,1]$, we have that $SLL(L^p[0,1], L^q[0,1]) = \{ 0 \}$ 
(see for example \cite[Ex.2.1]{caldel}).
But if $p > q$, $L^p[0,1] \subset L^q[0,1]$, and is it well known that $\mathcal M(L^p[0,1], L^q[0,1]) = L^r[0,1],$ where $1/p + 1/r = 1/q,$ so $SLL(L^p[0,1], L^q[0,1]) = \{ T_\Phi: \Phi \in L^r(\mu, Lip_0(\mathbb R)) \}$.
Thus, for the $L^p(\mu)$ spaces of finite measures, the condition $X(\mu) \subseteq Y(\mu)$ is always satisfied when there are more strong lattice Lispchitz operators than the null one.


\begin{remark}
Since $\Phi$ is strongly measurable we can approximate it pointwise with a sequence of simple functions
$\Phi_n(w)= \sum_{i=1}^{k_n} \varphi^n_i \chi_{A^n_i}(w) $ but in a $\mu-$null set $N.$
Fix a function $f \in X(\mu)$ and recall that we have assumed that $X(\mu) \subseteq Y(\mu).$
Then for every $w \in \Omega,$
\begin{equation*}
	|\Phi_n(w)(f(w))|= \left| \sum_{i=1}^{k_n} \varphi^n_i( f(w)) \chi_{A^n_i}(w) \right| \le \sum_{i=1}^{k_n} Lip(\varphi^n_i ) | f(w)| \chi_{A^n_i}(w) \in X(\mu),
\end{equation*}
and so $ \Phi_n(w)( f(w)) $ belongs to $Y(\mu).$ Note also that 
\begin{equation*}
	\sum_{i=1}^{k_n} Lip(\varphi^n_i ) | f(w)| \chi_{A^n_i}(w) = Lip( \Phi_n(w)) \, | f(w)|	
\end{equation*}
for every $w.$
Consider now the (non-negative) functions
\begin{equation*}
	h_n(\cdot) := \inf \big\{ Lip(\Phi_k(\cdot)) \, | f(\cdot)|: k \ge n \big\}.
\end{equation*} 
Since every Banach function space is order complete, the  functions $h_n$ belong to $Y(\mu).$ On the other hand,
$Lip(\Phi(w)- \Phi_n(w)) \to_n 0$  for  $w \in \Omega \setminus N.$ This gives
for  every  $w \notin N,$
\begin{equation*}
	\Big| Lip(\Phi(w)) \cdot f(w) - Lip(\Phi_n(w)) \cdot f(w) \Big| \le Lip(\Phi(w)- \Phi_n(w)) \cdot \big| f(w) \big| \to_n 0,
\end{equation*}
what implies that $h_n(\cdot) \to_n Lip(\Phi(\cdot)) \, | f(\cdot)|$ $\mu-$a.e.  The problem is that, in general, a sequence of measurable functions converging pointwise to a function in a Banach function space does not necessarily converges in norm to this function. 
So we need further requirements, both in the space and in the properties of the functions involved.
\end{remark}

\vspace{2pt}

\begin{theorem} \label{repFat}
Let $X(\mu)$ and $Y(\mu)$ be  
Banach function spaces such that  $X(\mu) \subseteq Y(\mu)$ and  $Y(\mu)$ has the Fatou property. The following facts are equivalent for  an operator $T:X(\mu) \to Y(\mu).$
\begin{itemize}

\item[(i)] $T$ is  strongly lattice Lipschitz.

\item[(ii)] There exist a strongly measurable function $\Phi:\Omega \to Lip_0(\mathbb R)$ such that $T=T_\Phi$ and 
  a sequence of essentially bounded  functions $(\Phi_n)_n \subset L_\infty(\mu, Lip_0(\mathbb R)),$ $\Phi_n: \Omega \to   Lip_0(\mathbb R)$ 
such that
$$
\sup_n \big\| Lip(\Phi_n(\cdot)) \big\|_{ \mathcal M ( X,Y)} < \infty
$$
and converges $\mu-$a.e. to $\Phi.$

\end{itemize}

Moreover, the sequence $(\Phi_n)_n$ can be chosen in such a way  that 
$$
\sup_n \big\| \Phi_n(\cdot) \big\|_{ \mathcal M ( X,Y)(\mu,Lip_0(\mathbb R) )} = \|\Phi\|_{ \mathcal M ( X,Y) (\mu,Lip_0(\mathbb R) )}.
$$


\end{theorem}

\vspace{0.3cm}
\begin{proof} (i) $\Rightarrow$ (ii)
If $T$ is strongly lattice Lipschitz it can be written as $T_\Phi$ for a certain strongly measurable function belonging to $\mathcal M(X,Y)(\mu, Lip_0(\mathbb R)).$ Also, there is a sequence of simple functions as $\Phi_n(w)= \sum_{i=1}^{k_n} \varphi^n_i \chi_{A^n_i}(w) $  that converges $\mu-$a.e. to $\Phi.$ Then for every $n \in \mathbb N$ consider the measurable set
\begin{equation*}
	B_n= \{ w :  Lip(\Phi_n(w)) < Lip(\Phi(w)) \} 
\end{equation*}
and the measurable function 
\begin{equation*}
	\psi_n(w):= \Phi_n(w)\chi_{B_n} + \Phi(w) \chi_{B_n^c}.
\end{equation*}

Clearly, 
\begin{equation*}
	Lip(\psi_n(\cdot)) = Lip(\Phi_n(\cdot)) \chi_{B_n} + Lip( \Phi(\cdot)) \chi_{B_n^c }\le Lip(\Phi_n(\cdot ))
\end{equation*}
$\mu-$a.e. and so it is  bounded $\mu-$a.e.
Also, by the ideal property of the norm of the Banach function space $\mathcal M(X,Y),$ we obtain the uniform bound for the norms $\| \psi_n \|_{\mathcal M(X,Y)(\mu, Lip_0(\mathbb R))} \le \| \Phi \|_{\mathcal M(X,Y)(\mu, Lip_0(\mathbb R))} .$ Finally, it is also clear that the sequence $(\psi_n)_n$ converges $\mu-$a.e. to $\Phi,$ what gives the result.

\vspace{4mm}

(ii) $\Rightarrow$ (i) 
By hypothesis, there is a strongly measurable  function $\Phi$ representing $T$ and we have to prove that 
$\Phi \in \mathcal M ( X,Y ) \big(\mu, Lip_0(\mathbb R) \big),$ that is   $\|\Phi(\cdot)\|_{Lip_0(\mathbb R)} = Lip(\Phi(\cdot)) \in \mathcal M(X,Y).$ 

\vspace{4mm}

By hypothesis, we can approximate $\Phi$ $\mu-$a.e.  with a sequence of simple functions
$\Phi_n$ that are  bounded but in a $\mu-$null set $N.$
Take a function $f \in X(\mu)$.
Since $X(\mu) \subseteq Y(\mu),$  we have for $\mu-$almost all $w$
$$
|\Phi_n(w)(f(w))| \le Lip( \Phi_n(w)) | f(w)| \le \| Lip( \Phi_n(w)) \|_{L^\infty(\mu)}  \,  | f(w)|  \in X(\mu) \subset Y(\mu),
$$
and so $ \Phi_n(w)( f(w)) $ belongs to $Y(\mu).$ 
Take the (non-negative classes of) functions 
$$
h_n(\cdot) := \inf \big\{ Lip(\Phi_k(\cdot)) \, | f(\cdot)|: k \ge n \big\}.
$$
Note that the infimum of every sequence of measurable functions that is bounded by below is again measurable, and every Banach function space is order complete, so for every $n \in \mathbb N,$ the (equivalence class of the) function $h_n$ belongs to $Y(\mu).$
Since we have  $Lip(\Phi(w)- \Phi_n(w)) \to_n 0$ $\mu-$a.e. we get
$$
\Big| Lip(\Phi(w)) \cdot f(w) - Lip(\Phi_n(w)) \cdot f(w) \Big| \le   Lip(\Phi(w)- \Phi_n(w)) \cdot \big| f(w) \big| \to_n 0,
$$
and so $h_n(\cdot) \to_n Lip(\Phi(\cdot)) \, | f(\cdot)|$ $\mu-$a.e. In addition,
the sequence $(h_n)_n$ is increasing and positive. Recall that  by hypothesis $\sup_n \|Lip(\Phi_n)(\cdot))\|_{M(X,Y)} < \infty,$ and so
\begin{align*}
	\sup_n \|h_n(\cdot) \|_{Y(\mu) }
	& \le \sup_n \|Lip(\Phi_n)(\cdot) \cdot | f(\cdot)|\|_{Y(\mu)} \\
	& \le  \sup_n \|Lip(\Phi_n)(\cdot)\|_{\mathcal M(X,Y)} \cdot  \|f\|_{X(\mu)} < \infty.
\end{align*}
Since $Y(\mu)$ has the Fatou property, we conclude  that the $\mu-$a.e.  limit of  the sequence defined by the functions $h_n$ (that is,  $Lip(\Phi(\cdot)) \, | f(\cdot)|$), belongs to $Y(\mu)$ too. Since $\Phi$ is strongly measurable, we know that $T_\Phi(f)$ that is given by (the equivalence class of) $\Phi(\cdot) ( f(\cdot))$ is a measurable function belonging to $L_0(\mu),$ and 
$$
 | \Phi(\cdot) (f(\cdot)) | \le Lip(\Phi(\cdot)) \, | f(\cdot)|.
 $$
Since we have shown that $Lip(\Phi(\cdot)) \, | f(\cdot)| \in Y(\mu),$ we obtain by the ideal property of $Y(\mu)$ that  $\Phi(\cdot) (f(\cdot)) \in Y(\mu).$ Note that all the arguments above are independent of the representative.
 This shows that the operator $T_\Phi$ is well-defined from $X(\mu)$ to $Y(\mu).$

Note that in fact we have shown more. We have proved that for every function $f\in X(\mu),$  the multiplication 
$Lip(\Phi(\cdot)) \,|f(\cdot)|$ belongs to $Y(\mu).$ Since a well-defined operator defined by the product with a  (positive) function is always continuous, we have that the multiplication 
operator $M_{Lip(\Phi)}: X(\mu) \to Y(\mu)$ given by the function $Lip(\Phi)(\cdot)$ is well-defined and, since it is positive, it is continuous (see \cite[Th.4.3]{ali}). This means that $Lip(\Phi(\cdot)) \in \mathcal M(X,Y).$ Since $\Phi$ is strongly measurable, we get that $\Phi \in \mathcal M(X,Y)( \mu, Lip_0(\mathbb R)),$ as desired.

\vspace{4mm}

Finally,   in order to prove the ``moreover part" of the result, consider the sequence  $(\psi_n)_n$ of vector valued functions and the associate sequence  of real valued measurable 
functions $(Lip( \psi_n(\cdot)))_n$ that has been constructed in the proof of (i) $\Rightarrow$(ii). 
The functions $\psi_n$ will play the role of the functions $\Phi_n$ appearing in the last statement of the result.
Since
 $(\psi_n)_n$ converges $\mu-$a.e. to $\Phi,$ we get that $(Lip( \psi_n(\cdot)))_n$ also converges $\mu-$a.e. to 
$Lip(\Phi(\cdot)).$ Consider now the sequence of real valued functions $(\tau_n)_n,$ where for each $n$
$$
\tau_n(\cdot)= \inf \big\{ Lip(\psi_k)(\cdot): k \ge n \big\}.
$$
Clearly, $\tau_n(\cdot)  \uparrow Lip(\Phi(\cdot)) $ $\mu-$a.e.  and, due to the ideal property of the space we also have
$\tau_n(\cdot) \in \mathcal M(X,Y)(\mu, Lip_0(\mathbb R)).$ Thus, we have that
$$
\big\| \tau_n(\cdot) \big\|_{\mathcal M(X,Y)}
\le 
\big\| \psi_n(\cdot) \big\|_{\mathcal M(X,Y)(\mu, Lip_0(\mathbb R))}
\le
\big\| \Phi(\cdot) \big\|_{\mathcal M(X,Y)(\mu, Lip_0(\mathbb R))}
$$
Since $Y(\mu)$ has the Fatou property we get that $\mathcal M(X,Y)$ has the Fatou property too (\cite[Prop. 3.3]{caldel}),
what gives that
$$
\sup_n \big\| \tau_n \big\|_{\mathcal M(X,Y)} =
\sup_n \big\| \psi_n \big\|_{\mathcal M(X,Y)(\mu, Lip_0(\mathbb R))}
=
\big\| \Phi \big\|_{\mathcal M(X,Y)(\mu, Lip_0(\mathbb R))}.
$$

%

\end{proof}

\vspace{0.6cm}

\section{Applications: approximation formulas for strongly lattice Lipschitz operators}

In the case of strongly lattice Lipschitz operators we can obtain a good representation in terms of tensor products (and consequently good approximation formulas) under some mild requirements. In this section we will show how to do so; we will also write down results for the case of $L^p-$spaces of finite measure to illustrate the procedure. These results will complement the analysis of lattice Lipschitz operators for the case of infinite measure spaces, which are known to be defined on Euclidean spaces and on $C(K)-$-spaces as a particular subclass of superposition operators. As we explained in the Introduction, the good representation of these operators playing the role of multiplication operators in the linear context is the key to use them as useful tools in analysis and to find good applications in other branches of applied mathematics such as Machine Learning.

Since we are working in the context of Köthe-Bochner spaces and operators between Banach function spaces, it is natural to consider tensor product representations for the space of strongly lattice Lipschitz operators.  Each simple tensor $t=\sum_{i=1}^n \varphi_i \otimes h_i \in Lip_0(\mathbb R) \otimes \mathcal M(X,Y)$ could be seen as an operator, in this case by pointwise composition, 
$$
T_t (f)(w) = \Big( \sum_{i=1}^n \varphi_i \otimes h_i  \Big)(f)(w) =  \sum_{i=1}^n \varphi_i \big(  f(w)  \big) \cdot h_i (w), \quad w \in \Omega.
$$

We can define  a sort of canonical representation of the tensor, based on the existence of standard representation of simple functions given by disjoint sums of single tensors as $t=\sum_{i=1}^n \varphi_i \otimes \chi_{A_i},$ what gives when acts on a function $f \in X(\mu)$
$$
T_t(f)(w)=  \sum_{i=1}^n \varphi_i \otimes \chi_{A_i}(w) (f(w))=  \sum_{i=1}^n \varphi_i \big( f(w) \big) \, \chi_{A_i}(w), \quad f \in X(\mu), \,\,\, w \in \Omega,
$$
for disjoint measurable sets $A_1,...,A_n.$ Write $S_{X,Y} $ for the subspace of the simple functions in $\mathcal M(X,Y)$
 (which can always be represented as above), 
and endow the tensor product $Lip_0(\mathbb R) \widehat\otimes  S_{X,Y} $ with the natural norm $\Delta_{\mathcal M(X,Y)}$ provided by the identification
$Lip_0(\mathbb R) \widehat\otimes  S_{X,Y} \ni t \mapsto T_t \in  SLL(X,Y).$ 
Assuming that simple functions are dense in $ \mathcal M(X,Y),$ we can obtain that the set of tensors as above is dense in
$$
Lip_0(\mathbb R) \widehat\otimes_{\Delta_{\mathcal M(X,Y)}}   \mathcal M(X,Y).
$$
This directly gives the following

\begin{corollary}
Let $X(\mu)$ and $Y(\mu)$ be  Banach function spaces such that simple functions are dense in $\mathcal M(X,Y) \big(\mu, Lip_0(\mathbb R) \big).$ Then
$$
Lip_0(\mathbb R) \widehat\otimes_{\Delta_{\mathcal M(X,Y)}}  \mathcal M(X,Y) =\mathcal M(X,Y) \big(\mu, Lip_0(\mathbb R) \big) = SLL(X,Y).
$$
Therefore, for each strongly lattice Lipschitz operator $T$ there is a sequence of simple tensors $t_n=  \sum_{i=1}^{k_n} \varphi^{k_n} \otimes \chi_{A_i^n} $ such that
$$
\lim_n \| T- \sum_{i=1}^{k_n} \varphi^{k_n} \otimes \chi_{A_i^n} \|_{\mathcal M(X,Y) (\mu, Lip_0(\mathbb R))} =0.
$$
\end{corollary}
\begin{proof}
The identification of  the closure of the tensor product with the space $\mathcal M(X,Y) \big(\mu, Lip_0(\mathbb R) \big)$ is immediate, due to the fact that the map $\sum_{i=1} \varphi_i \otimes \chi_{A_i}  \mapsto  \sum_{i=1} \varphi_i \chi_{A_i}$ is injective,
where $\sum_{i=1} \varphi_i \otimes \chi_{A_i}$ is any simple tensor in $Lip_0(\mathbb R) \otimes \mathcal M(X,Y)$ with $\{A_i\}_i$ disjoint measurable sets. Density of simple functions and coincidence of the norm for the set of simple functions in both spaces gives the equality. 

The identification of $ \mathcal M(X,Y) \big(\mu, Lip_0(\mathbb R) \big)$ with $ SLL(X,Y)$ is a consequence of requirements in Definition
\ref{defSSL} and  the injectivity of the map
that carries each function $\Phi \in \mathcal M(X,Y) \big(\mu, Lip_0(\mathbb R) \big)$ to the operator $T_\Phi$ that is defined in Proposition \ref{th1strong}. The injectivity is also given by the last statement in this proposition. This proves the result.

\end{proof}

It is well known that simple functions are dense in Bochner spaces $L^p(\mu,E)$ where $E$ is any Banach space, provided that $1 \le p < \infty$ (see for example \cite[S.7 .2]{defa}, which are particular cases of multiplication operator spaces such as those we are concerned with for $E=Lip_0(\mathbb R). $ However, if $E$ is an infinite dimensional Banach space, simple functions are not dense in $L^\infty(\mu,E);$ in fact, the tensor product $L^\infty(\mu) \otimes E$ is not dense in $L^\infty(\mu,E)$ (\cite[Ex.4 .9]{defa}). Let us write as usual $\Delta_p$ for the norm on the tensor product of $Lip_0(\mathbb R)$ and  $L^p(\mu)$ inherited from the Bochner space $L^p(\mu,Lip_0(\mathbb R).$ Recall that for $\mu$ finite, $1\le q \le p < \infty$ and $r$ such that $1/q=1/p+1/r,$ we have that  $L^p(\mu) \subseteq L^q(\mu)$ and $\mathcal M(L^p(\mu),L^q(\mu))= L_r(\mu).$

\begin{corollary}
Let $\mu$ be a finite measure, $1\le q \le p  < \infty$ and $1/q=1/p+1/r.$  Then an operator $T:L^p(\mu) \to L^q(\mu)$ is strongly lattice Lipschitz if and only if 
there is  a function $\Phi \in L^r \big(\mu, Lip_0(\mathbb R) \big)$ such that
$$
T(f)(w)= \Phi(w) \circ f(w) \quad \text{for} \,\, w \in \Omega \,\,\, \text{$\mu-$a.e} \,\,\text{ and for all } \,\,\,\, f \in X(\mu).
$$
In other words, we have the isometries
$$
 Lip_0(\mathbb R)  \widehat \otimes_{\Delta_r}    L^r(\mu) = L^r \big(\mu, Lip_0(\mathbb R) \big)= SLL(L^p,L^q).
$$
\end{corollary}

Using Theorem \ref{repFat} we can easily get the next result on $L_p-$spaces over finite measures.  The spaces $L^p(\mu)$ for $1\le p < \infty$ have the Fatou property. In accordance with the initial purpose of the paper, next corollary shows the similarities with the description of multiplication operators between $L^p-$spaces.

\begin{corollary}
Let $\mu$ be a finite measure, $1\le q \le p  < \infty$ and $1/q=1/p+1/r.$  Then an operator $T:L^p(\mu) \to L^q(\mu)$ is strongly lattice Lipschitz if and only if there exists  a sequence of simple  functions $(\Phi_n)_n \subset L^r \big(\mu, Lip_0(\mathbb R) \big)$   that converges pointwise $\mu-$a.e. to the representing function $\Phi(w)(\lambda)= T(\lambda \chi_\Omega)(w)$ and such that
$$
\sup_n \big\| Lip(\Phi_n(\cdot)) \big\|_{ L^r(\mu)} < \infty.
$$
\end{corollary}

Although $\Delta_p$ is not a tensor norm, we always have that the inequalities $\varepsilon \le \Delta_r \le \pi$ hold for tensor products as the one above, where $\epsilon $ is the injective tensor norm and $\pi$ the projective one (see \cite[Ch.7]{defa}). For example,  there is a surjection from the tensor product $F \widehat\otimes_\pi E'$ onto the nuclear operators $\mathcal N(E,F)$ between $E$ and $F.$ On the other hand, it is also known that the dual of the Arens-Ells space $\AE(\mathbb R)$ is $ Lip_0(\mathbb R).$
Taking into account the inequality $\Delta_r \le \pi$ we have that
we can identify every nuclear operator $N: \AE(\mathbb R) \to L^r(\mu)$ with a strongly lattice Lipschitz operator $T_N:L^p(\mu) \to L^q(\mu)$ through the isomorphisms written above.

Let us finish the paper with some comments on the suitable applications of the results in pure analysis. It is well-known that multiplication operators are the key in the contexts of the spaces of integrable functions to obtain good diagonalization settings for operators between infinite dimensional function spaces. For example, they allow to get some fundamental tools for the geometry of function spaces, as the circle of ideas that permits to relate lattice geometric properties of function spaces (p-convexity, p-concavity, Boyd indexes,...) with factorization of operators through $L_p-$spaces. Some classical results, as the famous Rosenthal Theorem, as well as the factorization theorems of Krivine, Maurey, Nikishin, Pietsch, Pisier and many  others are in the center of, for example, summability theory for operators or the geometry of the Banach spaces. We are then intended to search for the generalization of the above mentioned results in the context of the Lipschitz operators, for which it seems to be necessary to know the main properties of the particular class of superposition operators which we call lattice Lipschitz operators.

A lot of questions are still open regarding the results presented in this paper.  Although we have shown some representation tools, we believe that it is possible to obtain more general results. We write below two open questions.

\begin{itemize}

\item Which requirements are needed to assure that all lattice Lipschitz operators are strongly lattice Lipschitz? Which are the spaces---others than the trivial ones---for which this holds?

\item Are all the strongly lattice Lipschitz operators between Banach function spaces $X(\mu)$ and $Y(\mu)$ representable by strongly measurable functions in $\mathcal M(X,Y)$ with weaker requirements on $Y(\mu)$? 

\end{itemize}

\section*{Ethics declarations}

\paragraph{\bf Conflict of interest} On behalf of all authors, the corresponding author states that there is no conflict of interest.

\end{document}